\documentclass[12pt]{amsart}

\usepackage{amsmath,amssymb,enumerate,mathtools}
\newtheorem{theorem}{Theorem}[section]
\newtheorem{claim}{}[theorem]
\newtheorem{lemma}[theorem]{Lemma}

\newtheorem{proposition}[theorem]{Proposition}
\newtheorem{corollary}[theorem]{Corollary}

\theoremstyle{definition}

\newcommand{\bF}{\mathbb F}
\newcommand{\bR}{\mathbb R}
\newcommand{\bZ}{\mathbb Z}

\newcommand{\cC}{\mathcal{C}}

\newcommand{\cM}{\mathcal{M}}

\newcommand{\Fprime}{\mathbb{F}_{\mathrm{prime}}}

\DeclareMathOperator{\rowspace}{rowspace}

\DeclareMathOperator{\dist}{dist}

\DeclareMathOperator{\GF}{GF}

\DeclareMathOperator{\sign}{sign}

\newcommand{\del}{\!\setminus\!}
\newcommand{\con}{/}
\begin{document}
\sloppy

\title[Asymptotically Good Codes]{On the existence of asymptotically good linear codes in minor-closed classes}

\author{Peter Nelson, Stefan H.M. van Zwam}

\maketitle

\begin{abstract}
  Let $\mathcal{C} = (C_1, C_2, \ldots)$ be a sequence of codes such that each $C_i$ is a linear $[n_i,k_i,d_i]$-code over some fixed finite field $\bF$, where $n_i$ is the length of the codewords, $k_i$ is the dimension, and $d_i$ is the minimum distance. We say that $\mathcal{C}$ is \emph{asymptotically good} if, for some $\varepsilon > 0$ and for all $i$, $n_i \geq i$, $k_i/n_i \geq \varepsilon$, and $d_i/n_i \geq \varepsilon$. Sequences of asymptotically good codes exist. We prove that if $\mathcal{C}$ is a class of $\GF(p^n)$-linear codes (where $p$ is prime and $n \geq 1$), closed under puncturing and shortening, and if $\mathcal{C}$ contains an asymptotically good sequence, then $\mathcal{C}$ must contain \emph{all} $\GF(p)$-linear codes. Our proof relies on a powerful new result from matroid structure theory.
\end{abstract}

\section{Introduction}

  For a finite field $\bF$, let $\Fprime$ denote the unique subfield of $\bF$ of prime order. For a linear code $C$, denote the length of $C$ by $n_C$, the dimension of $C$ by $k_C$, and the minimum Hamming distance of $C$ by $d_C$. In short, $C$ is an $[n_C, k_C, d_C]$ code. A class $\cC$ of codes is \emph{asymptotically good} if there exists $\varepsilon > 0$ such that for every $n \in \bZ^+$ there is a code $C \in \cC$ of length $n_C \geq n$ satisfying $k_C \ge \varepsilon n_C$ and $d_C \ge \varepsilon n_C$. 
  
  For every finite field $\bF$, the class of linear codes over $\bF$ is asymptotically good, as suitable random codes have nonvanishing rate and minimum distance. Our main result can be seen as a converse to this statement. Our result involves two standard operations on linear codes. Given a code $C$, the \emph{puncturing} of $C$ at $i$ is the code obtained from $C$ by removing the $i$th coordinate from each word. The \emph{shortening} of $C$ at $i$ is the code obtained from $C$ by selecting only the codewords of $C$ having a $0$ in position $i$, and then puncturing the resulting code at $i$.
  
  \begin{theorem}\label{main}
    Let $\bF$ be a finite field. If $\cC$ is an asymptotically good class of linear codes over $\bF$, then every linear code $C'$ over $\Fprime$ can be obtained from a code $C \in \cC$ by a sequence of puncturings and shortenings. 
  \end{theorem}
  
  In other words, the only asymptotically good classes of $\bF$-linear codes that are closed under puncturings and shortenings are those that contain all linear codes over some field. This result was conjectured by Geelen, Gerards, and Whittle [\ref{ggw}]. A restatement of the above theorem for prime fields is the following: 
\begin{theorem}
  Let $\bF$ be a field of prime order. If $\cC$ is a proper subclass of the linear codes over $\bF$ that is closed under puncturing and shortening, then $\cC$ is not asymptotically good. 
\end{theorem}
  When $|\bF| = 2$, this substantially generalises results of Kashyap [\ref{kashyap}] which show that the classes of graphic binary codes, as well as a slightly larger class of `almost-graphic' codes, are not asymptotically good.
  
  Our proof makes fundamental use of a deep theorem in structural matroid theory recently announced by Geelen, Gerards and Whittle. This theorem is one of many outcomes of the `matroid minors project', the result of more than a decade of work generalising Robertson and Seymour's graph minors structure theorem [\ref{rs}] to matroids representable over finite fields. While this theorem has now been stated in print [\ref{ggw}], its proof will stretch to hundreds of pages and has yet to be published. The reader should be aware of this contingency; for a more detailed discussion see [\ref{ggw}].

\section{Preliminaries}
  Since the ideas in our proof are matroidal, we adopt the terminology of matroid theory. The correspondence between linear codes and matroids is fairly direct even when a matroid is defined in the usual way by its ground set and rank function, but for convenience we will deal with a flavour of matroid whose definition coincides exactly with that of a linear code.
  
  \subsection*{Represented matroids} If $\bF$ is a field, then an \emph{$\bF$-represented matroid} is a pair $M = (E,U)$, where $E$ is a finite set and $U$ is a subspace of $\bF^E$. We often omit `$\bF$-represented' when the context is clear. We write $|M|$ for $|E|$.
  
  If $A$ is an $\bF$-matrix with column set $E$ such that $\rowspace(A) = U$, then we write $M = M(A)$; we call $A$ a \emph{generator matrix} for $M$ and say that $A$ \emph{generates} $M$. For $X \subseteq E$ we write $u|X$ for the restriction of the vector $u$ to those coordinates indexed by $X$, we write $U|X$ for the space $\{u|X: u \in U\}$, and we write $r_M(X)$ for the dimension of $U|X$. We denote $r_M(E)$ simply by $r(M)$. We call $r_M$ the \emph{rank function} of $M$.

If $C$ is a $q$-ary $[n,k,d]$ code, then $M = (\{1,\ldots, n\}, C)$ is a matroid with $|M| = n$ and $r(M) = k$. The rank function simply gives the dimension of this code and of all its puncturings.
  
  If $A_2$ is an $\bF$-matrix obtained from an $\bF$-matrix $A_1$ by nonzero column scalings, then $M(A_1)$ and $M(A_2)$ are formally distinct but share the same rank function. We give a name to this equivalence: two $\bF$-represented matroids $(E,U_1)$ and $(E,U_2)$ are \emph{projectively equivalent} if $U_2 = \{uD: u \in U_1\}$ for some nonsingular diagonal matrix $D$. 
  
  \subsection*{Matroid terminology} Matroid duality coincides with linear code duality. The \emph{dual matroid} $M^*$ of an $\bF$-represented matroid $M = (E,U)$ is defined to be $(E,U^{\perp})$, where $U^\perp$ denotes the orthogonal complement of $U$. For $X \subseteq E$ we write $M \del X$ for $(E-X,U|(E-X))$ and $M \con X$ for $(M^* \del X)^*$; these are the matroids obtained from $M$ by \emph{deletion} and \emph{contraction} of $X$ respectively; these operations correspond to puncturing and shortening of codes. If $N$ is an $\bF$-represented matroid that is projectively equivalent to $M \con C \del D$ for some disjoint subsets $C$ and $D$ of $E$, then we say $N$ is a \emph{minor} of $M$. A class $\cM$ of matroids is \emph{minor-closed} if $\cM$ is closed under taking minors and isomorphism. 
    
   A set $X \subseteq E$ is \emph{spanning} in $M$ if $r_M(X) = r(M)$ and \emph{independent} in $M$ if $r_M(X) = |X|$. If $X$ is not independent then $X$ is \emph{dependent}. A \emph{circuit} of $M$ is a minimal dependent set and a \emph{cocircuit} of $M$ is a minimal dependent set of $M^*$, or equivalently a minimal set $C$ satisfying $r(M \del C) < r(M)$. For a matroid $M$, we write $g(M)$ for the size of a smallest circuit of $M$, also called the \emph{girth} of $M$.

If $C$ is a $q$-ary $[n,k,d]$ code, then $M = (\{1,\ldots, n\}, C)$ is a matroid with $g(M^*) = d$. For that reason we sometimes write $d(M)$ for $g(M^*)$.
   
  \subsection*{Connectivity}
  A notion fundamental in matroid theory that arises in our proof is that of connectivity. Informally, a matroid has low connectivity if it can be obtained by `gluing' two smaller matroids together on a low-dimensional subspace. There are many notions of matroid connectivity but we just need one; for $t \in \bZ^+$, a matroid $M = (E,U)$ is \emph{vertically $t$-connected} if, for every partition $(A,B)$ of $E$ satisfying $r_M(A) + r_M(B) < r(M) + t - 1$, either $A$ or $B$ is spanning in $M$. For instance, $M$ is vertically 2-connected if and only if $M$ cannot be written as the direct sum of two positive-rank matroids. Note that other authors often use a more restrictive definition of ``vertically $t$-connected,'' which implies ours.
   
  \subsection*{Frame matroids}
   An \emph{$\bF$-frame matrix} is an $\bF$-matrix in which every column has at most two nonzero entries, and an \emph{$\bF$-represented frame matroid} is a matroid having an $\bF$-frame matrix as a generator matrix. For a group $\Gamma$, a \emph{$\Gamma$-labelled digraph} is a pair $(G,\Sigma)$, where $G = (V,E)$ is a directed graph (allowing loops and multiple edges) and $\Sigma: E \to \Gamma$ is an assignment of a label in $\Gamma$ to every arc of $G$. There is a well-known and natural correspondence between $\bF$-frame matroids and $\bF^{\times}$-labelled digraphs; a full treatment is given in [\ref{ox2}] and a reader familiar with these concepts can skip this subsection, where we just give the minimum definitions and observations we will need.  
   
   If $A$ is an $\bF$-frame matrix with row set $V$ and column set $E$, then a \emph{graph representation} of $A$ is an $\bF^{\times}$-labelled digraph $(G,\Sigma)$, where $G = (V,E)$ and $(G,\Sigma)$ satisfies the following conditions: 
  \begin{itemize}
    \item If $A[e]$ has two nonzero entries in rows $x$ and $y$, then $e$ is an arc of $G$ from $x$ to $y$ with label $-A[e,x]A[e,y]^{-1}$ or an arc of $G$ from $y$ to $x$ with label $-A[e,y]A[e,x]^{-1}$, and 
    \item If $A[e]$ has exactly one nonzero entry, then $e$ is a loop of $G$ at $x$ with arbitrary label in $\bF^\times - \{1\}$. 
    \item If $A[e] = 0$, then $e$ is a loop of $G$ with label $1$. 
   \end{itemize}
   
   It is clear that one frame matrix may have many graph representations, and that graph representations always exist unless $|\bF| = 2$, where a frame matrix having a column with exactly one nonzero entry has no graph representations since $\bF^{\times}-\{1\}$ is empty. However, appending a `parity' row to a binary frame matrix yields another row-equivalent frame matrix where every column has even support. Since we can remove redundant rows from an arbitrary frame matrix to still have a frame matrix and append a single row in this way in the binary case, we have the following statement, which we apply freely.
   
   \begin{proposition}\label{frameexist}
    If $M$ is an $\bF$-represented frame matroid, then there is a generator matrix $A$ of $M$ having a graph representation and at most $r(M) + 1$ rows.
   \end{proposition}
   
   A \emph{cycle} or \emph{path} of a digraph $G$ will denote any cycle or path of the underlying undirected graph of $G$. Let $C$ be a cycle of $G$, and $v_1,e_1,v_2,e_2,\dotsc,e_{k-1},v_k=v_1$ be a corresponding alternating sequence of vertices and arcs of $G$ (there are two choices for this sequence). Let $C^+$ be the set of $e_i$ such that $e_i$ is directed from $v_i$ to $v_{i+1}$ in $G$, and $C^-$ be the set of all other $e_i$. The \emph{sign} of $C$ (relative to this sequence) is $\prod_{e \in C^+} \Sigma(e) \prod_{e\in C^-} \Sigma(e)^{-1}$. We say that $C$ is a \emph{balanced} cycle of $(G,\Sigma)$ if $\sign(C) = 1$. Note that a cycle being balanced does not depend on the choice of sequence of vertices and edges.
   There is a well-known characterisation of the set of circuits of $M(A)$ in terms of the balanced and unbalanced cycles of $(G,\Sigma)$, but we will just use the following weaker statement, which is fairly straightforward to check by considering linear dependencies in the columns of $A[C]$. 
     
   \begin{proposition}\label{bicycledependent}
    If $(G,\Sigma)$ is a graph representation of an $\bF$-frame matrix $A$ and $C$ is a balanced cycle of $(G,\Sigma)$ or a connected subgraph of $G$ of minimum degree 2 which is not a cycle of $G$, then $C$ is dependent in $M(A)$. 
   \end{proposition}
   
  We say $(G, \Sigma')$ was obtained from $(G, \Sigma)$ by \emph{resigning} if, for some $\gamma \in \Gamma$ and for some partition $(U,W)$ of the vertices of $G$, we have
  \begin{align*}
    \Sigma'(e) = \begin{cases}
        \gamma \Sigma(e) & \text{ if } e \text{ runs from } U \text{ to } W;\\
        \gamma^{-1} \Sigma(e) & \text{ if } e \text{ runs from } W \text{ to } U;\\
        \Sigma(e) & \text{ otherwise}.
    \end{cases}
  \end{align*} 
  It is easily checked that $(G, \Sigma)$ and $(G, \Sigma')$ have the same collection of balanced cycles. In the representation this corresponds to scaling the rows indexed by $W$ by a factor $\gamma$.

   In what follows, we need to assume that the graph representation of a frame matroid is connected, which means that there is a path between every pair of vertices. It is easy to show that a vertically $3$-connected matroid, with no elements $e$ such that $r_M(e) = 0$, has the property that every pair of elements is in a circuit. A consequence of this is:

   \begin{lemma}\label{lem:graphconn}
     Let $(G,\Sigma)$ be a graph representation of an $\bF$-represented frame matroid $M$. If $M$ is vertically 3-connected and has no loops, then $G$ is connected.
   \end{lemma}

   \begin{corollary}\label{cor:graphconn}
     Let $M$ be an $\bF$-represented frame matroid. If $M$ is vertically 3-connected, then $M$ has a graph representation $(G,\Sigma)$ with $G$ connected.
   \end{corollary}

  \subsection*{Asymptotically good matroids} Finally, we redefine asymptotic goodness, this time for matroids. For $\alpha,\beta \in \bR$ we say a sequence $M_0,M_1 \dotsc,$ of matroids is \emph{$(\alpha,\beta)$-good} if $|M_i| \ge i$, $r(M_i) \ge \alpha |M_i|$ and $g(M_i^*) \ge \beta|M_i|$ for each $i \in \bZ_0^+$. A class $\cM$ of matroids is \emph{asymptotically good} if $\cM$ contains an $(\alpha,\beta)$-good sequence for some $\alpha,\beta \in \bR^+$. Note that $\alpha,\beta \le 1$ for any such $\alpha,\beta$. For any finite field $\bF$, the class of  all $\bF$-represented matroids is such a class.

  \section{Connectivity}
  Our goal in this section is to show that, to prove our main result, it suffices to focus on highly connected matroids.
    
\begin{lemma}\label{connreduction}
  If $t \in \bZ$ and $\cM$ is an asymptotically good minor-closed class of matroids, then the class of vertically $t$-connected matroids in $\cM$ is asymptotically good. 
\end{lemma}
\begin{proof}
  For each $\beta \in \bR^+$ let $A_{\beta}$ denote the set of all $\alpha \in \bR^+$ such that $\cM$ contains an $(\alpha,\beta)$-good sequence.
  Let $B = \{\beta \in \bR^+: A_\beta \ne \varnothing\}$. By assumption, $B$ is a nonempty interval with $\inf(B) = 0$ and $\sup(B) \le 1$; let $\beta_{\max} = \sup(B)$. Each nonempty $A_x$ is also such an interval; for each $x \in B$, let $\alpha_{\max}(x) = \sup(A_x)$, noting that $\alpha_{\max}(x) \le 1$. 
  
  Let $\beta \in (\tfrac{2}{3}\beta_{\max},\beta_{\max})$ and let $\delta = \tfrac{\beta}{4(1-\beta)}$. Set $\alpha_{\max} = \alpha_{\max}(\beta)$ and let $\alpha \in ((1+\tfrac{1}{2}\delta)^{-1}\alpha_{\max},\alpha_{\max})$. We have $\beta \in B$ and $\alpha \in A_{\beta}$; let $M_0, M_1, \dotsc$ be an $(\alpha,\beta)$-good sequence of matroids in $\cM$. Since $2\beta > \beta_{\max}$ we have $2\beta \notin B$ so $A_{2\beta} = \varnothing$ and in particular $\tfrac{1}{2}\alpha \notin A_{2\beta}$. Moreover, we have $(1+\delta)\alpha > \alpha_{\max}$ so $(1+\delta)\alpha \notin A_\beta$. There are therefore integers $m_1$ and $m_2$ such that and every matroid $M \in \cM$ with $|M| \ge m_1$ satisfies $g(M^*) < 2\beta |M|$ or $r(M) < \tfrac{1}{2}\alpha|M|$, and every matroid $M \in \cM$ with $|M| \ge m_2$ satisfies $g(M^*) < \beta |M|$ or $r(M) < \alpha(1+\delta)|M|$. Let $m = \max(m_1,m_2,8t\alpha^{-1})$. 
  
  \begin{claim}
    For each $n \in \bZ$ with $n \ge 2\beta^{-1}m$, the matroid $M_n$ is vertically $t$-connected. 
  \end{claim}
  \begin{proof}
    Suppose for a contradiction that $n \ge 2\beta^{-1}m$ and $M = M_n$ is not vertically $t$-connected. Let $(X_1,X_2)$ be a partition of $E(M)$ with $r_M(X_1) + r_M(X_2) = r(M) + t' - 1$ with $t' < t$ and $r_M(X_1),r_M(X_2) \ge t'$. Let $N_1 = M\con X_2$ and $N_2 = M \con X_1$. Note that $N_1,N_2 \in \cM$ and that $r(N_i) = r(M)-r_M(X_{3-i}) > r(M)-t+1$, so $r(N_1) + r(N_2) > r(M)-2t$. Since every cocircuit of $N_1$ or $N_2$ is a cocircuit of $M$, we have $g(M^*) \le \min(g(N_1^*),g(N_2^*))$.We have
        \[ \beta|M| \le g(M^*) \le \min(g(N_1^*),g(N_2^*)) \le \tfrac{1}{2}(g(N_1^*) + g(N_2^*))\]
    and, since $|M| = |N_1| + |N_2|$, we have either $g(N_1^*) \ge 2\beta|N_1|$ or $g(N_2^*) \ge 2\beta|N_2|$. We may assume that the first case holds. Since each of $X_1$ and $X_2$ contains a cocircuit of $M$, we have $|N_i| = |X_i| \ge g(M^*) \ge \beta|M| \ge m$ for each $i \in \{1,2\}$. Since $|N_1| + |N_2| = |M|$, this implies that $|N_1| \ge \tfrac{\beta}{1-\beta}|N_2|$. Moreover, $m \ge m_1$ gives $r(N_1) < \tfrac{1}{2}\alpha|N_1|$. We have 
    \begin{align*}
      \alpha|M| &\le r(M) \\
            &< r(N_1) + r(N_2) + 2t\\
            &< \tfrac{1}{2}\alpha|N_1| + r(N_2) + 2t\\
            &\le \tfrac{3}{4}\alpha|N_1| + r(N_2),
    \end{align*}
    Where the last line uses $\tfrac{1}{4}\alpha|N_1| \ge \tfrac{1}{4}\alpha m \ge 2t$. From this and $|M| = |N_1|+|N_2|$ we get 
    \begin{align*}
      r(N_2) &\ge \tfrac{1}{4}\alpha|N_1| + \alpha|N_2|\\
           &= \alpha\left(1 + \delta \right)|N_2|
    \end{align*} 
    Now $|N_2| \ge m \ge m_2$ so $g(N_2^*) < \beta |N_2|$ by choice of $m_2$. However $g(M^*) \le g(N_2^*)$ and $\beta|N_2| < \beta|M|$, so $g(M^*) < \beta|M|$, contradicting the definition of $M$. 
  \end{proof}
  By the claim, all but finitely many terms of the sequence $M_0,M_1, \dotsc$ are vertically $t$-connected, so the class of vertically $t$-connected matroids in $\cM$ is asymptotically good, as required. 
\end{proof}

\section{The Structure Theorem}\label{structuresection}

  For each field $\bF$ of prime characteristic $p$, we write $\Fprime$ for the unique subfield of $\bF$ with $p$ elements. In this section we state the deep structural result on which our proof is based. Essentially the theorem states that for any minor-closed class $\cM$ of $\bF$-represented matroids not containing all $\Fprime$-represented matroids, the highly connected members of $\cM$ are `close' to being an $\bF$-represented frame matroid or its dual. We need to define our notion of distance. 
  
  Our distance metric is based on `lifts' and `projections'. If $M_1 = (E,U_1)$ and $M_2 = (E,U_2)$ are $\bF$-represented matroids and there is an $\bF$-represented matroid $M$ with ground set $E \cup \{e\}$ satisfying $M \del e = M_1$ and $M \con e = M_2$, then we say that $M_2$ is an \emph{elementary projection} of $M_1$ and $M_1$ is an \emph{elementary lift} of $M_2$. For arbitrary $\bF$-represented matroids $M_1$ and $M_2$ on a common ground set $E$, we write $\dist(M_1,M_2)$ for the minimum number of elementary lifts/projections required to transform $M_1$ into $M_2$. (It is clear that any matroid on $E$ can be transformed into the rank-$0$ matroid on $E$ by a finite sequence of projections, so this distance is always finite.) It is easy to see that, if $\dist(M_1,M_2) \le k$, then there is a matroid $M$ with ground set $E \cup C \cup D$ satisfying $M \del D \con C = M_1$ and $M \con D \del C = M_2$, where $|C|+|D| \le k$. Each elementary lift and projection can change the rank of a subset by at most one. From this we can deduce the following lemma.

  \begin{lemma}\label{lem:connectedperturb}
  	Let $M, N$ be $\bF$-represented matroids with $\dist(M,N) \leq k$. If $M$ is vertically $t$-connected, then $N$ is vertically $(t - 2k)$-connected.
  \end{lemma}
  
  The structure theorem, a weakened combination of Lemma 2.1 and Theorems~3.2 and 3.3 in [\ref{ggw}], can now be stated.  
  
  \begin{theorem}\label{big}
    Let $\bF$ be a finite field and let $\cM$ be a minor-closed class of $\bF$-represented matroids not containing all projective geometries over $\Fprime$. There exists $k \in \bZ^+$ such that every vertically $k$-connected matroid $M$ in $\cM$ satisfies $\dist(M,N) \le k$ or $\dist(M^*,N) \le k$ for some $\bF$-represented frame matroid $N$. 
  \end{theorem}

\section{Small Circuits}

  In this section we show that if $M$ is a rank-$r$ matroid with has significantly more than $r$ elements and $M$ or $M^*$ is close to a frame matroid or its dual, then the girth of $M$ is at most logarithmic in $r$. The primal case is slightly more difficult and will result from the following corollary of a result of Alon, Hoory and Linial [\ref{alonetal}] observed by Kashyap [\ref{kashyap}]:
    
  \begin{lemma}\label{moore}
    If $G$ is a graph with girth $g$ and average degree $\delta > 2$, then $g \le 4 + \tfrac{\log(|V(G)|)}{\log(\delta - 1)}$.
  \end{lemma}
  
  The next three results extend the above lemma to matroids that are close to frame matroids. As before, a \emph{cycle} of a graph refers strictly to a set of its edges. 
  
  \begin{lemma}\label{disjointcircuits}
    Let $t \in \bZ^+$ and $\beta \in \bR^+$. If $G$ is a graph with $|V(G)| =  n\ge \max\left(\left(\frac{4t}{\beta\log(1+\beta)}\right)^2,(1+\beta)e^4\right)$ and $|E(G)| \ge (1+\beta)n$, then $G$ has a collection of $t$ pairwise edge-disjoint cycles, each of size at most $\tfrac{2\log n}{\log(1+\beta)}$. 
  \end{lemma}
  \begin{proof}
    Let $m = |E(G)|$ and $\alpha = 2(\log(1+\beta))^{-1}$. By choice of $n$ we have $\frac{t \alpha \log n}{n} < \tfrac{t \alpha}{\sqrt{n}} < \tfrac{1}{2}\beta$.
     Let $\cC$ be a maximal collection of pairwise disjoint cycles of $G$ such that $|E(C)| \le \alpha \log n$ for each $C \in \cC$. Assume for a contradiction that $|\cC| < t$; let $G' = G \del \cup \cC$. We have $\tfrac{|E(G')|}{n} \ge \tfrac{m - t\alpha \log(n)}{n} \ge 1+\beta - \tfrac{t \alpha \log(n)}{n} > 1 + \tfrac{1}{2}\beta$, so the average degree $\delta'$ of $G'$ is at least $2 + \beta$. By maximality of $\cC$, the graph $G'$ has girth $g' > \alpha \log (n)$, so Lemma~\ref{moore} gives \begin{align*}
      \alpha \log(n)  &< 4 + \frac{\log(n)}{\log(\delta'-1)}  \le 4 + \frac{\log(n)}{\log(1+\beta).}
\end{align*}
  Rearranging gives $\log(n) < 4\log(1+\beta)$, contradicting $n \ge (1+\beta)e^4$.
  \end{proof}
  
  \begin{corollary}\label{framematroidrankdeficiency}
    Let $\bF$ be a finite field of order $q$, let $\beta \in \bR^+$ and $t \in \bZ^+$. Let $M$ be an $\bF$-represented frame matroid with graph representation $(G,\Sigma)$, such that $G$ is connected. If $M$ satisfies \[r(M) \ge \max\left(|\bF|,(2q-3)\beta^{-1}+1-q,\tfrac{1}{q-1}\left(\tfrac{4t}{\beta\log(1+\beta)}\right)^2,\tfrac{1+\beta}{q-1}e^4\right)\] and $|M| \ge (1+q\beta)r(M)$, then there is a set $X \subseteq E(M)$ such that $r_M(X) \le |X|-t$ and $|X| \le \tfrac{4 t\log r(M)}{\log(1+\beta)}$.
  \end{corollary}
  \begin{proof}
    Let $A$ be an $\bF$-frame matrix generating $M$ and let $(G,\Sigma)$ be a graph representation of $A$. Pick a spanning tree $T$ of $G$. By repeatedly resigning, we may assume that the edges of $T$ have sign $1$. Note that $|E(T)| \geq r(M) - 1$.
    
     Let $G^+$ denote the undirected graph with vertex set $V(G) \times \bF^\times$ and edge set 
     \begin{align*}
      & \{((\delta^{-}(e),\gamma),(\delta^{+}(e),\gamma)): e \in T, \gamma \in \bF\}\\
      \cup & \{((\delta^-(e), 1), (\delta^+(e), \sign(e))): e \in E(G)\setminus T\};
     \end{align*}
      that is, we take $q-1$ copies of $T$, and each directed edge $(u,v)$ not in $T$ with sign $\gamma$ connects the copy of $T$ corresponding to $1$ with the copy of $T$ corresponding to $\gamma$.
It is easy to check that each cycle of $G^+$, by projection onto the first coordinate, corresponds to either a balanced cycle of $(G,\Sigma)$, or a subgraph of $G$ of minimum degree 2 that is not a cycle. It follows from Proposition~\ref{bicycledependent} that every cycle of $G^+$ is dependent in $M$. Now,
\begin{align*}
    |E(G^+)| & = (q-1)|E(T)| + |M| - |E(T)| \geq (q-2)(r(M) - 1) + |M|\\
    & \ge (q-1)(1+\beta)r(M) + \beta r(M) - (q-2)\\
    & \ge (1+\beta)|V(G^+)| - (q-1)(1+\beta) + \beta r(M) - q + 2\\
    &  \ge (1+\beta)|V(G^+)|,  
\end{align*}

     where we use $|E(T)| \geq r(M)-1$ first, $|M| \geq (1+q\beta)r(M)$ second, $|V(G^+)| \le (q-1)(r(M)+1)$ third, and $r(M) \ge (2q-3)\beta^{-1}+1-q$ last. By our lower bounds for $n = |V(G+)| \geq (q-1)r(M)$ and Lemma~\ref{disjointcircuits}, the graph $G^+$ contains $t$ pairwise disjoint cycles, each of size at most $\frac{2\log n}{\log(1+\beta)} \le \frac{4 \log r(M)}{\log(1+\beta)}$. Each of these cycles of $G$ is dependent in $M$, and it is thus easy to see that their union satisfies the result. 
  \end{proof}
  
  \begin{lemma}\label{nearframe}
    Let $k \in \bZ^+$, $\beta \in \bR^+$, and $\bF$ be a finite field. There exists $c \in \bZ$ such that if $M$ is a rank-$r$ $\bF$-represented matroid with $r > 2$ and $|M| \ge (1+ \beta)r$ and there is an $\bF$-represented, connected frame matroid $N$ satisfying $\dist(M,N) \le k$, then $g(M) \le c \log r$.
  \end{lemma}
  \begin{proof}
    Let $\beta' = \tfrac{1}{2}\beta$ and $c \ge \max\Big((2q-3)\beta^{-1}+1-q$, $|\bF|$, $\left(\tfrac{4(k+1)}{\beta'\log(1+\beta')}\right)^2$, $(1+\beta')e^4\Big)+k$ be an integer so that $4(k+1)\log(x+k) \le c \log(1+\beta')\log x$ for all $x \ge c$. Since $\beta' > 0$ we know that $M$ has a circuit, so $g(M) \le r+1$. If $r \le c-1$ then $g(M) \le r+1 \le c \le c \log r$, so the result holds; we may thus assume that $r \ge c$.
    
    Let $M^+$ be a matroid and $C,D \subseteq E(M^+)$ be sets of size at most $k$ so that $M^+ \con C \del D = M$ and $M^+ \con D \del C = N$. Since each $Y \subseteq E(M)$ satisfies $r_{M^+}(Y) \ge r_M(Y) \ge r_{M^+}(Y) - r_{M^+}(C) \ge r_{M^+}(Y) - k$ and a similar statement holds for $N$, we have  $|r_M(Y) - r_{N}(Y)| \le k$ for each $Y \subseteq E(M)$. In particular, $r+k \ge r(N) \ge r - k \ge c-k$. By choice of $c$, Corollary~\ref{framematroidrankdeficiency} implies there is a set $X \subseteq E(N)$ with $r_{N}(X) \le |X|-(k+1)$ and $|X| \le \frac{4(k+1)\log(r(N))}{\log(1+\beta')} \le \tfrac{4(k+1)\log(r+k)}{\log(1+\beta')} \le c \log r$. Now $r_M(X) \le r_{N}(X) + k < |X|$, so $X$ contains a circuit of $M$ and therefore $g(M) \le c \log r$, as required. 
  \end{proof}
  
  We now deal with the case when $M$ is close to the dual of a frame matroid. Here we show that the girth is bounded above by a constant.
  
  \begin{lemma}\label{nearcoframe}
    Let $k \in \bZ^+$, $\beta \in \bR^+$, and $\bF$ be a finite field. There exists $c \in \bZ$ so that, if $M$ is a nonempty $\bF$-represented matroid such that  $|M| \ge (1+\beta) r(M)$ and there is an $\bF$-represented frame matroid $N$ with $\dist(M^*,N) \le k$ and connected graph, then $g(M) \le c$.  
  \end{lemma}
  \begin{proof}
    We may assume that $\beta \le 1$. Let $c = \lceil 12 \beta^{-1}(3k+1) \rceil$. If $|M| \le c$ then the result clearly holds, so we will assume otherwise. Let $M^+$ be a matroid and $C,D \subseteq E(M^+)$ be sets of size at most $k$ so that $M^+ \con C \del D = M^*$ and $M^+ \con D \del C = N$. As before, we have $|r_{M^*}(Y)-r_N(Y)| \le k$ for each $Y \subseteq E(M^*)$. Let $A$ be an $\bF$-frame matrix generating $N$ with $r(N)$ rows (or $r(N)+1$ rows if $|\bF| = 2$), and $(G,\Sigma)$ be an $\bF^{\times}$-labelled graph associated with $N$.  We have $|M^*| = |M| \ge (1+\beta)r(M) = (1+\beta)(|M^*| - r(M^*))$, giving $|M^*| \le (1 + \beta^{-1})r(M^*) \le 2\beta^{-1}r(M^*)$.  
    
    Now $k \le r(M^*)-k \le r(N) \le |V(G)|+1 \le 2|V(G)|$, so 
    \[|E(G)| = |N| = |M^*| \le 2\beta^{-1}(|V(G)|+k) \le 6\beta^{-1}|V(G)|.\] 
    Therefore $G$ has average degree at most $12\beta^{-1}$. If $|V(G)| < 2(k+1)$, then $2(k+1)-1 \ge r(N) \ge r(M^*)-k$ so $r(M^*) \le 3k+1$. This gives $|M^*| \le 2\beta^{-1}(3k+1) < c$, a contradiction.  Therefore $|V(G)| \ge 2(k+1)$, so there is a collection of $2(k+1)$ vertices of $G$ whose degrees sum to at most $12\beta^{-1}(k+1)$, so there is a set $F \subseteq E(G)$ such that $|F| \le 12\beta^{-1}(k+1) \le c$ and $r(N \del F) \le r(N)+1 - 2(k+1)$ (since $A$ has at most $r(N)+1$ rows and $A[E(N)-F]$ has at least $2(k+1)$ zero rows). Now
    \[r(M^* \del F) \le r(N \del F) + k \le r(N) - 2k -1 +k \le r(M^*) -1.\]
    Thus $F$ contains a cocircuit of $M^*$, which is a circuit of $M$, giving $g(M) \le c$. 
    
  \end{proof}
  
  \section{The Main Result} 
  
  The following theorem implies Theorem~\ref{main}.
  
  \begin{theorem}
    Let $\bF$ be a finite field. If $\cM$ is a minor-closed class of $\bF$-represented matroids, then $\cM$ is asymptotically good if and only if $\cM$ contains all projective geometries over $\Fprime$.
  \end{theorem}
  \begin{proof}
    If $\cM$ contains all projective geometries over $\Fprime$, then $\cM$ contains all $\Fprime$-represented matroids so is clearly asymptotically good. Suppose that $\cM$ does not contain all projective geometries over $\Fprime$ but is asymptotically good. By Theorem~\ref{big}, there is an integer $k$ so that every vertically $k$-connected matroid $M \in \cM$ satisfies $\dist(M,N) \le k$ or $\dist(M^*,N) \le k$ for some $\bF$-represented frame matroid $N$; moreover, $N$ is vertically $3$-connected by Lemma \ref{lem:connectedperturb}, so by Corollary \ref{cor:graphconn} we may assume the associated graph is connected. Let $\cM_k$ denote the class of vertically $(2k+3)$-connected matroids in $M$. Note that every such matroid is also vertically $k$-connected.

    By Lemma~\ref{connreduction}, the class $\cM_k$ is asymptotically good, so $\cM_k$ contains an $(\alpha,\alpha)$-good sequence for some $\alpha \in (0,1)$. Let $\beta = (1-\alpha)^{-1}-1$ and let $c$ be the maximum of the two integers given by Lemmas~\ref{nearframe} and~\ref{nearcoframe} for $\bF$, $\beta$ and $k$. Let $n_0$ be an integer so that $c \log n < \alpha n$ for all $n \ge n_0$. 
    
    There is a matroid $M \in \cM_k$ so that $|M| \ge n_0$, $d(M) = g(M^*) \ge \alpha|M|$ and $r(M) \ge \alpha|M|$. The last inequality gives $|M^*| \ge (1+\beta)r(M^*)$, so by Lemma~\ref{nearframe} or~\ref{nearcoframe} we have $d(M) \le c \log r(M^*) \le c \log |M^*| < \alpha |M^*|$, a contradiction. 
    
    
    
  \end{proof}
\section*{Acknowledgements}
	We thank Navin Kashyap for his careful reading and very useful advice on the manuscript. 

\section*{References}

\newcounter{refs}

\begin{list}{[\arabic{refs}]}
{\usecounter{refs}\setlength{\leftmargin}{10mm}\setlength{\itemsep}{0mm}}


\item\label{alonetal}
N. Alon, S. Hoory and N. Linial, The Moore bound for irregular graphs, Graphs Combin. 18 (2001), 53--57. 

\item\label{kashyap}
N. Kashyap, A decomposition theory for binary linear codes, IEEE Transactions on Information Theory 54 (2008), 3035--3058.

\item\label{ggw}
J. Geelen, B. Gerards and G. Whittle, The highly connected matroids in minor-closed classes, arXiv:1312.5012 [math.CO].

\item\label{ox2}
J. G. Oxley, 
Matroid Theory (2nd edition),
Oxford University Press, New York, 2011.

\item\label{rs}
N. Robertson and P. D. Seymour, Graph Minors. XVI. Excluding a non-planar
graph, J. Combin. Theory Ser. B 89 (2003), 43--76.

\end{list}

\end{document}